\documentclass[runningheads]{llncs}
\usepackage{amsfonts}
\usepackage{amscd,color}
\usepackage{amsmath,amsfonts,amssymb,amscd}
\usepackage{indentfirst,graphicx,epsfig}
\usepackage{graphicx}
\usepackage{graphicx}
\usepackage{epstopdf}
\input{epsf}
\usepackage{graphicx}
\usepackage{epstopdf}
\usepackage{caption}

\newtheorem{Question}{Question}

\begin{document}

\title{Complexity Results for Two Kinds of Colored Disconnections of Graphs\thanks{Supported by NSFC No.11871034 and 11531011.}}

\author{You Chen \and Ping Li \and Xueliang Li \and Yindi Weng}

\authorrunning{ Y. Chen et al.}

\institute{ Center for Combinatorics and LPMC\\
  Nankai University, Tianjin 300071, China\\
\email{chen\_you@163.com; wjlpqdxs@163.com;\\
         lxl@nankai.edu.cn; 1033174075@qq.com}\\}

\maketitle

\begin{abstract}
The concept of rainbow disconnection number of graphs was introduced by Chartrand et al. in 2018. Inspired by this concept, we put forward the concepts of rainbow vertex-disconnection and proper disconnection in graphs. In this paper, we first show that it is $NP$-complete to decide whether a given edge-colored graph $G$ with maximum degree $\Delta(G)=4$
is proper disconnected. Then, for a graph $G$ with $\Delta(G)\leq 3$ we show that $pd(G)\leq 2$ and determine the graphs with $pd(G)=1$ and $2$, respectively. Furthermore, we show that for
a general graph $G$, deciding whether $pd(G)=1$ is $NP$-complete, even if $G$ is bipartite. We also show that it is $NP$-complete to decide whether a given vertex-colored graph $G$ is rainbow
vertex-disconnected, even though the graph $G$ has $\Delta(G)=3$ or is bipartite.

\end{abstract}

\section{Introduction}

\noindent All graphs considered in this paper are finite, simple and undirected. Let $G=(V(G),E(G))$ be a nontrivial
connected graph with vertex set $V(G)$ and edge set $E(G)$.
For a vertex $v\in V$, the \emph{open neighborhood} of $v$ is the set $N(v)=\{u\in V(G) | uv\in E(G)\}$ and the \emph{degree} of $v$ is
$d(v)=|N(v)|$, and the \emph{closed neighborhood} is the set $N[v]=N(v)\cup \{v\}$.
For any notation and terminology not defined here, we follow those used in \cite{BM}.

For a graph $G$ and a positive integer $k$, let $c: E(G)\rightarrow[k]$ ($c: V(G)\rightarrow[k]$) be an edge-coloring (vertex-coloring) of $G$, where and in what follows $[k]$
denotes the set $\{1,2,...,k\}$ of integers.
For an edge $e$ of $G$, we denote the color of $e$
by $c(e)$.

In graph theory, paths and cuts are two dual concepts. By Menger's Theorem, paths are in the same position as cuts are in studying graph connectivity. Chartrand et al. in \cite{CJMZ} introduced the concept
of \emph{rainbow connection} of graphs. \emph{Rainbow disconnection}, which is a dual concept of rainbow connection, was introduced by Chartrand et al. \cite{CDHHZ}.
An \emph{edge-cut} of a graph $G$ is a set $R$ of edges such that
$G-R$ is disconnected. If any two edges in $R$ have different colors, then $R$
is a \emph{rainbow edge-cut}. An edge-coloring is called a
\emph{rainbow disconnection coloring} of $G$ if for every
two distinct vertices of $G$, there exists a rainbow edge-cut in
$G$ separating them. For a connected graph $G$, the \emph{rainbow disconnection number}
of $G$, denoted by $rd(G)$, is the smallest number of
colors required for a rainbow disconnection coloring of
$G$. A rainbow disconnection coloring using $rd(G)$ colors is called an
rd-\emph{coloring} of $G$. In \cite{BCL} and \cite{BHL} the authors
obtained many results on the rainbow disconnection number.

Inspired by the concept rainbow disconnection, the authors in \cite{BCLLW} and \cite{LW} introduced the concept of rainbow vertex-disconnection.
For a connected and vertex-colored graph $G$, let $x$ and $y$ be two vertices of $G$.
If $x$ and $y$ are nonadjacent, then an $x$-$y$ \emph{vertex-cut} is a subset $S$ of $V(G)$ such that $x$ and $y$ belong to
different components of $G-S$. If $x$ and $y$ are adjacent, then an $x$-$y$ \emph{vertex-cut} is a subset $S$ of $V(G)$ such that
$x$ and $y$ belong to different components of $(G-xy)-S$. A vertex subset $S$ of $G$ is \emph{rainbow} if no two vertices of $S$
have the same color. An $x$-$y$ \emph{rainbow vertex-cut} is an $x$-$y$ vertex-cut $S$ such that if $x$ and $y$ are nonadjacent,
then $S$ is rainbow; if $x$ and $y$ are adjacent, then $S+x$ or $S+y$ is rainbow.

A vertex-colored graph $G$ is called \emph{rainbow vertex-disconnected} if for any two distinct vertices $x$ and $y$ of $G$,
there exists an $x$-$y$ rainbow vertex-cut. In this case, the vertex-coloring $c$ is called a \emph{rainbow vertex-disconnection coloring}
of $G$. For a connected graph $G$, the \emph{rainbow vertex-disconnection number}
of $G$, denoted by $rvd(G)$, is the minimum number of colors that are needed to make $G$ rainbow vertex-disconnected.
A rainbow vertex-disconnection coloring with $rvd(G)$ colors is called an
rvd-\emph{coloring} of $G$.

Andrews et al.~\cite{ALLP} and Borozan et al.~\cite{BFGMMMT} independently introduced the concept of \emph{proper connection} of graphs.
Inspired by the concept of rainbow disconnection and proper connection of graphs, the authors in \cite{BCJ} and \cite{CLLW} introduced the concept of proper
disconnection of graphs. For an edge-colored graph $G$, a set $F$ of edges of $G$ is a \emph{proper edge-cut}
if $F$ is an edge-cut of $G$ and any pair of adjacent edges in $F$ are assigned by different colors. For any two vertices $x,y$ of $G$,
an edge set $F$ is called an $x$-$y$ proper edge-cut if $F$ is a proper edge-cut and $F$ separates $x$ and $y$ in $G$. An edge-colored
graph is called \emph{proper disconnected} if for each pair of distinct vertices of $G$ there exists a proper edge-cut separating them.
For a connected graph $G$, the \emph{proper disconnection number} of $G$, denoted by $\mathit{pd(G)}$, is defined as the minimum number
of colors that are needed to make $G$ proper disconnected, and such an edge-coloring is called a pd-\emph{coloring}. From \cite{BCJ},
we know that if $G$ is a nontrivial connected graph, then
$1\leq \mathit{pd}(G)\leq \mathit{rd}(G)\leq \chi'(G)\leq \Delta(G)+1$, where $\chi'(G)$ denotes the chromatic index or edge-chromatic number
of $G$.

These graph parameters are some kinds of chromatic numbers, which are used to characterize the global property, i.e., the connectivity for colored graphs.
An immediate question is how to calculate them ? Are there any good or efficient algorithms to compute them ? or it is NP-hard to get them.
For the rainbow disconnection number of graphs, we showed in \cite{BCL} that it is NP-complete to determine whether the rainbow disconnection
number of a cubic graph is 3 or 4, and moreover, we showed that given an edge-colored graph $G$ and two vertices $s, t$ of $G$, deciding
whether there is a rainbow cut separating $s$ and $t$ is NP-complete. In this paper we will determine the computational complexity of
proper (rainbow vertex-) disconnection numbers of graphs.

Our paper is organized as follows. In Section 2, we first show that it is $NP$-complete to decide whether a given $k$-edge-colored graph $G$
with maximum degree $\Delta(G)=4$ is proper disconnected. Then for a graph $G$ with $\Delta(G)\leq 3$, we show that $pd(G)\leq 2$, and determine
the graphs with $pd(G)=1$ and $2$, respectively. Furthermore, we show that it is $NP$-complete to decide whether $pd(G)=1$, even though $G$ is a bipartite graph.
In Section 3, we show that it is $NP$-complete to decide whether a given vertex-colored graph $G$ is rainbow vertex-disconnected,
even though the graph $G$ has $\Delta(G)=3$ or is bipartite.

\section{Hardness results for proper disconnection number}

\noindent In this section, we show that it is $NP$-complete to decide
whether a given $k$-edge-colored graph $G$ with maximum degree $\Delta(G)=4$ is proper disconnected.
Then we give the proper disconnection numbers of graphs with $\Delta(G)\leq 3$, and propose an unsolved question.
Furthermore, we show that it is $NP$-complete to decide whether $pd(G)=1$, even though $G$ is a bipartite graph.

\subsection{Hardness results for graphs with maximum degree four}

\noindent We first give some notations. For an edge-colored graph $G$,
let $F$ be a proper edge-cut of $G$.
If $F$ is a matching, then $F$ is called a \emph{matching cut}.
Furthermore, if $F$ is an $x$-$y$ proper edge-cut for vertices $x,y\in G$,
then $F$ is called an $x$-$y$ \emph{matching cut}.
For a vertex $v$ of $G$, let $E_v$ denote the set of all edges incident with $v$ in $G$.

We can obtain the following results by means of
a reduction from the NAE-3-SAT problem. At first we present the NAE-3-SAT problem, which is $NP$-complete; see~\cite{sch78,ADJD}.

{\bf Problem: } Not-All-Equal 3-Sat (NAE-3-SAT)

{\bf Instance: } A set $C$ of clauses, each containing 3 literals from a set of boolean variables.

{\bf Question:} Can truth value be assigned to the variables so that each clause contains at least one true literal and at least one false literal ?

Given a formula $\phi$ with variable $x_1,\cdots,x_n$,
let $\phi=c_1\wedge c_2\wedge\cdots\wedge c_m$,
where $c_i=(l_1^i\vee l_2^i\vee l_3^i)$.
Then $l_j^i\in\{x_1,\overline{x_1},\cdots,x_n,\overline{x_n}\}$
for each $i\in[m]$ and $j\in[3]$.

\begin{figure}[h]
    \centering
    \includegraphics[width=234pt]{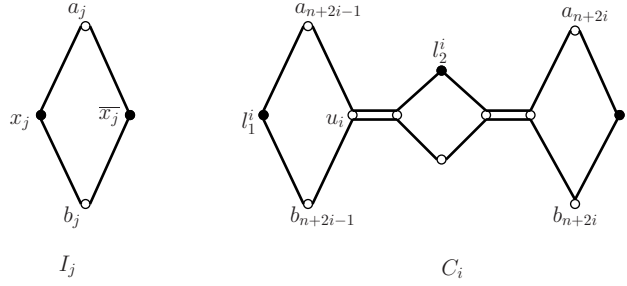}\\
    \caption{The graphs $I_j$ and $C_i$.} \label{SAT1}
\end{figure}

We will construct a graph $G_\phi$ below.
The graphs $I_j$ and $C_i$ are shown in Figure \ref{SAT1} where $j\in[n]$ and $i\in[m]$.
Each graph $C_i$ has two pairs of parallel edges.
The graph $G_\phi$ (see Figure \ref{SAT11}) is obtained from mutually disjoint graphs $I_j$ and $C_i$
by adding a pair of parallel edges between $z$ and $w$
if $z,w$ satisfy one of the following conditions:
\begin{enumerate}
 \item $z=a_i$ and $w=a_{i+1}$ for some $i\in[n+2m-1]$;
 \item $z=b_i$ and $w=b_{i+1}$ for some $i\in[n+2m-1]$;
 \item $z=x_j$, $w=l_t^i$ and $x_j=l_t^i$ for some $j\in[n], t\in[3]$ and $i\in[m]$;
 \item $z=\overline{x_j}$, $w=l_t^i$ and $\overline{x_j}=l_t^i$ for some $j\in[n], t\in[3]$ and $i\in[m]$.
\end{enumerate}

\begin{figure}[h]
    \centering
    \includegraphics[width=320pt]{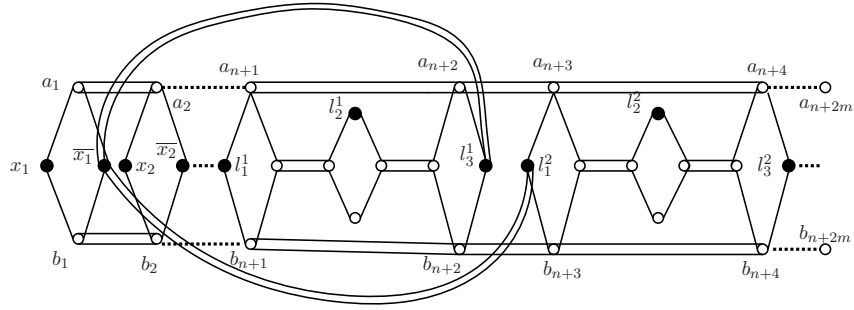}\\
    \caption{The graph $G_\phi$ with $l_3^1=l_1^2=\overline{x_1}$.} \label{SAT11}
\end{figure}

In fact, the graph $G_\phi$ was constructed in \cite{PP} (in Section 3.2).
It is obvious that each vertex of $G_\phi$ with degree greater than four
is a vertex with even degree. Moreover, there are two simple edges incident with this kind of vertex,
and the other edges incident with the vertex are some pairs of parallel edges.
The authors proved that $G_\phi$ has a matching cut if and only if
the corresponding instance $\phi$ of NAE-3-SAT problem has a solution.

We present a {\em star structure} as shown in Figure \ref{Star}\ (1).
Each vertex $z_i$ is called a {\em tentacle}.
A star structure is a {\em $k$-star structure}
if it has $k$ tentacles.
\begin{figure}[h]
    \centering
    \includegraphics[width=350pt]{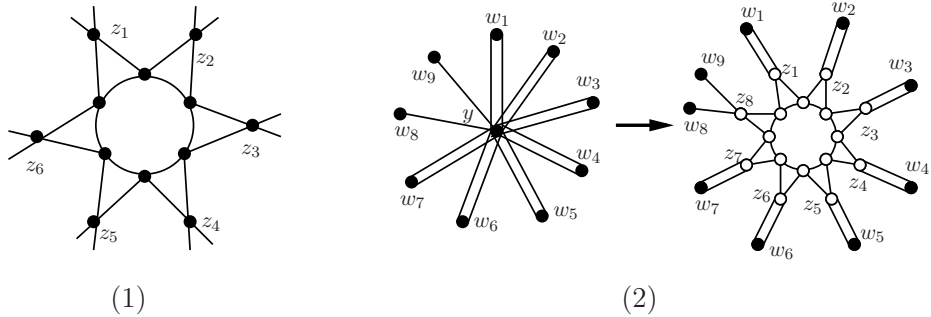}\\
    \caption{(1) A $6$-star structure with tentacles $z_1,\cdots,z_6$, and (2) the operation $\mathcal{O}$ on vertex $y$ with degree $16$.} \label{Star}
\end{figure}

For a vertex $y$ of $G_\phi$ with $d_{G_\phi}(y)=2t+2>4$,
assume $N(y)=\{w_1,\cdots,w_{t+2}\}$
such that $w_{t+1},w_{t+2}$ connect $y$ by a simple edge respectively,
and $w_i$ connects $y$ by a pair of parallel edges for $i\in[t]$.
Now we define an operation $\mathcal{O}$ on vertex $y$:
replace $y$ by a $(t+1)$-star structure with tentacles $z_1,\cdots,z_{t+1}$
such that $w_{i}$ and $z_{t+1}$ for $i\in \{t+1,t+2\}$ are connected by a simple edge, and $z_i$ and $w_i$ are connected by parallel edges for $i\in[t]$.
As an example, Figure \ref{Star} (2) shows the operation $\mathcal{O}$ on vertex $y$ with degree $16$.
We apply the operation $\mathcal{O}$ on each vertex of degree greater than four,
and then subdivide one of each pair of parallel edges by a new vertex in $G_\phi$.
Denote the resulting graph by $G'_\phi$, which is a simple graph.
The graph $G'_\phi$ was also defined in \cite{PP},
and the authors showed that $G'_\phi$ has a matching cut if and only if
the corresponding instance $\phi$ of NAE-3-SAT problem has a solution.

Now we construct a graph, denoted by $H_\phi$, obtained from $G_\phi$ by operations as follows.
Add two new vertices $u$ and $v$.
Connect $u$ and each vertex of $\{a_1,a_{n+2m}\}$
by a pair of parallel edges, and connect $v$
and each vertex of $\{b_1,b_{n+2m}\}$ by a pair of parallel edges.
We apply the operation $\mathcal{O}$ on each vertex of degree greater than four in $H_\phi$,
and then subdivide one of each pair of parallel edges by a new vertex.
Denote the resulting graph by $H'_\phi$\ (see Figure \ref{SAT22}), which is a simple graph.
Observe that $\Delta(H'_\phi)=4$.
Since a minimal matching cut cannot contain any edge in a triangle,
we know that there is a $u$-$v$ matching cut in $H'_\phi$ if and only if
there is a matching cut in $G'_\phi$. Thus, there is a $u$-$v$ matching cut in $H'_\phi$ if and only if
the instance $\phi$ of NAE-3-SAT problem has a solution.
\begin{figure}[h]
    \centering
    \includegraphics[width=300pt]{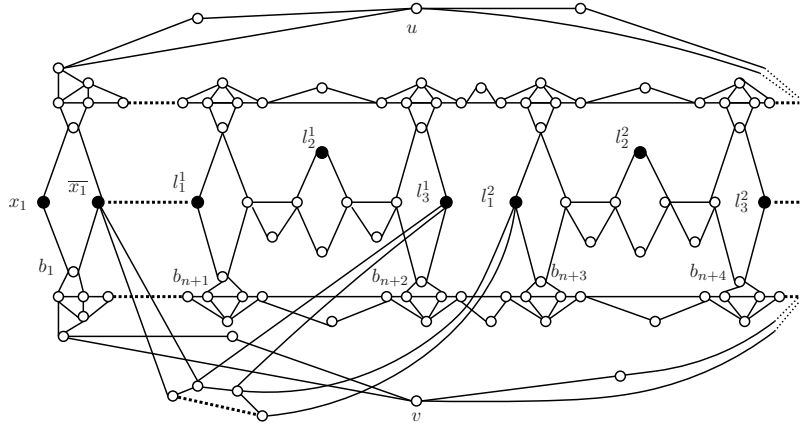}\\
    \caption{The graph $H'_\phi$ with $l_3^1=l_1^2=\overline{x_1}$.} \label{SAT22}
\end{figure}

\begin{theorem} \label{NPC}
For a fixed positive integer $k$, let $G$ be a $k$-edge-colored graph with maximum degree $\Delta(G)=4$, and let $u,v$ be any two specified vertices of $G$.
Then deciding whether there is a $u$-$v$ proper edge-cut in $G$ is $NP$-complete.
\end{theorem}
\begin{proof}
For a connected graph $G$ with an edge-coloring $c: E(G)\rightarrow [k]$
and an edge-cut $D$ of $G$,
let $M_i=\{e\ |\  e\in D \mbox{ and } c(e)=i\}$ for $i\in [k]$.
Then $D$ is a proper edge-cut if and only if each $M_i$ is a matching.
Therefore, deciding whether a given edge-cut of an edge-colored graph is a proper edge-cut is in $P$.

For an instance $\phi$ of the NAE-3-SAT problem,
we can obtain the corresponding graph $H'_\phi$ as defined above.
Then there is a vertex, say $y'$, of $H'_\phi$ with degree two.
Let $G$ be a graph obtained from $H'_\phi$ and a path $P$ of order $k$ by identifying $y'$ and one of the ends of $P$.
Then $\Delta(G)=4$. We color each edge of $G-E(P)$ by $1$ and color $k-1$ edges of $P$ by $2,3,\cdots,k$, respectively.
Then the edge-coloring is a $k$-edge-coloring of $G$, and there is a $u$-$v$ proper edge-cut in $G$ if and only if there is a $u$-$v$ matching cut in $H'_\phi$.
Thus, we get that there is a $u$-$v$ proper edge-cut in $G$ if and only if
the instance $\phi$ of NAE-3-SAT problem has a solution.
\end{proof}

\begin{corollary}\label{cor}
For a fixed positive integer $k$, let $G$ be a $k$-edge-colored graph with maximum degree $\Delta(G)=4$. Then
it is $NP$-complete to decide whether $G$ is proper disconnected.
\end{corollary}

Even though it is still not clear for the computational complexity of deciding whether a graph with maximum degree at most three is proper disconnected,
we will show that $pd(G)\leq 2$ for a graph $G$ with maximum degree $\Delta(G)\leq 3$ and then determine the graphs with $pd(G)=1$ and $2$, respectively.
Some preliminary results are given as follows, which will be used in the sequel.

\begin{theorem}\upshape\cite{BCJ}\label{treepd}
If $G$ is a tree, then $pd(G)=1$.
\end{theorem}

\begin{theorem}\upshape\cite{BCJ}\label{cyclepd}
If $C_n$ be a cycle, then
$$\textnormal{pd}(C_n)=\left\{
\begin{array}{lcl}
2, &  {if~n=3},\\
1, &  {if~n\geq 4.}
\end{array}
\right.$$
\end{theorem}

\begin{theorem}\upshape\cite{BCJ}\label{t2}
For any integer $n\geq 2$, $pd(K_{n})=\lceil \frac{n}{2} \rceil$.
\end{theorem}

\begin{theorem}\upshape\cite{BCJ}\label{pd1}
Let $G$ be a nontrivial connected graph. Then $pd(G)=1$
if and only if for any two vertices
of $G$, there is a matching cut separating them.
\end{theorem}

\begin{theorem}\upshape\cite{BM}\label{perfectmatching} (Petersen's Theorem)
Every $3$-regular graph without cut edges has a perfect matching.
\end{theorem}

For a simple connected graph $G$, if $\Delta(G)=1$, then $G$ is the graph $K_2$, a single edge.
If $\Delta(G)=2$, then $G$ is a path of order $n\geq 3$ or a cycle. By Theorems \ref{treepd} and \ref{cyclepd},
for a connected graph $G$ with $\Delta(G)\leq 2$, we have $pd(G)=1$ if and only if $G$ is a path or a cycle of order $n\geq 4$,
and $pd(G)=2$ if and only if $G$ is a triangle.

Next, we will present the proper disconnection numbers of graphs with maximum degree $3$. At first we give the proper disconnection
numbers of $3$-regular graphs.
\begin{lemma}\label{lpdnc}
If $G$ is a $3$-regular connected graph without cut edges, then $pd(G)\leq 2$.
\end{lemma}

\begin{figure}[htbp]
\centering
\includegraphics[width=0.3\textwidth]{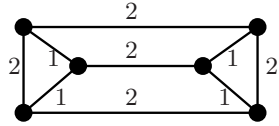}
\caption{The graph $G_0$}\label{figpd3regno}
\end{figure}

\begin{proof}
Let $G_0$ be a graph by connecting two triangles with $3$ matching (or independent) edges, and we color $G_0$
with two colors as shown in Figure \ref{figpd3regno}. Obviously, it is a proper disconnection coloring of $G_0$.
Now we consider $2$-edge-connected $3$-regular graphs $G$ except $G_0$. By Theorem \ref{perfectmatching},
there exists a perfect matching $M$ in $G$. We define an edge-coloring $c$ of $G$ as follows.
Let $c(M)=2$. If $E(G)\setminus M$ contains triangles, then we color one of the edges in each triangle by color $2$.
We then color the remaining edges by color $1$. Since $G\setminus M$ is the union of some disjoint cycles,
we denote these disjoint cycles by $C_1,C_2,\cdots C_t$. Let $x,y$ be two vertices of $G$.
If $x$ and $y$ belong to different cycles of $C_1,C_2,\cdots C_t$, then $M$ is an $x$-$y$ proper edge-cut.
If $x$ and $y$ belong to the same cycle $C_i$ ($i\in [t]$), then there are two cases to discuss.

\textbf{Case 1.} $|C_i|\geq 4$.

Since $|C_i|\geq 4$, there exist two $x$-$y$ paths $P_1,P_2$ in $C_i$. We choose two nonadjacent
edges $e_1, e_2$ respectively from $P_1, P_2$. Then $M\cup \{e_1,e_2\}$ is an edge-cut separating $x$ and $y$.
Since $c(M)=2$ and $c(e_1)=c(e_2)=1$, $M\cup \{e_1,e_2\}$ is an $x$-$y$ proper edge-cut.

\textbf{Case 2.} $|C_i|=3$.

Since $x,y\in C_i$, we can assume $C_i=xyz$. Let $N(x)=\{y,z,x_0\}$ and $N(x_0)=\{x,x_1,x_2\}$. Assume $x_0\in C_k,k\in [t]\setminus\{i\}$.

{\em Subcase 2.1.} $c(xy)=1$.

Assume $c(yz)=1$ and $c(xz)=2$. Note that $x_1\notin N(z)$ or $x_2\notin N(z)$,
without loss of generality, say $x_2\notin N(z)$.

For $|C_k|\geq 4$, we have $c(x_0x_2)=1$.
Then $E_{x_2}\setminus \{x_0x_2\}$ have different colors. So, $\{xy,xz,x_0x_1\}\cup E_{x_2}\setminus \{x_0x_2\}$ is an $x$-$y$ proper edge-cut.

For $|C_k|=3$, if $c(x_0x_1)\neq c(x_0x_2)$, we get that $\{xy,xz,x_0x_1,x_0x_2\}$ is an $x$-$y$ proper edge-cut.
Now consider $c(x_0x_1)=c(x_0x_2)=1$. If $x_1\in N(z)$, then $c(x_1z)=2$. Since $G\neq G_0$, we have $x_2\notin N(y)\cup N(z)$. So, $(E_y\setminus \{yz\})\cup\{xz,x_0x_1,x_1x_2\}$ is an $x$-$y$
proper edge-cut. If $x_1\notin N(z)$, then denote $E_{x_1}\setminus \{x_0x_1,x_1x_2\}$ by $e_1$ and denote $E_{x_2}\setminus \{x_0x_2,x_1x_2\}$ by $e_2$.
It is clear that $e_1,e_2\in M$. So, $c(e_1)=c(e_2)=2$. We get that $\{xy,xz,e_1,e_2\}$ is an $x$-$y$ proper edge-cut.

{\em Subcase 2.2.} $c(xy)=2$.

In Subcase $2.1$, if $c(xy)=1$, then the $x$-$y$ proper edge-cut is also an $x$-$z$ proper edge-cut. So, we have proved Subcase $2.2$.
\end{proof}

Let $H(v)$ be a connected graph with one vertex $v$ of degree two and the remaining vertices of degree three. We assume that
the neighbors of $v$ in $H(v)$ are $v_1$ and $v_2$, respectively. If $v_1$, $v_2$ are adjacent, then denote it by $H_1(v)$.
Otherwise, denote it by $H_2(v)$. Let $H_1'(v)$ be the graph obtained by replacing the vertex $v$ by a diamond.
Let $H_2'(v)$ be the graph obtained by replacing the path $v_1vv_2$ of $H_2(v)$ by an new edge $v_1v_2$; see Figure \ref{gpro}.
\begin{figure}[htbp]
\centering
\includegraphics[width=0.7\textwidth]{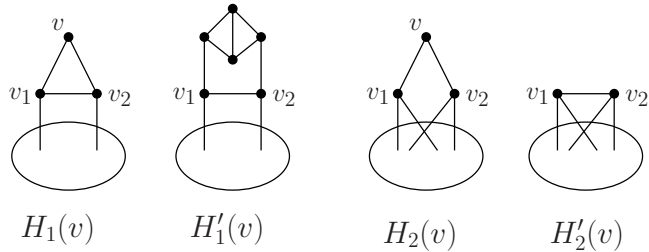}
\caption{The graph process}
\label{gpro}
\end{figure}

\begin{lemma}\label{pd3regular}
If $G$ is a $3$-regular graph of order $n$ $(n\geq 4)$, then $pd(G)\leq 2$.
\end{lemma}

\noindent $Proof.$ We proceed by induction on the order $n$ of $G$. Since a $3$-regular graph of order $4$ is $K_4$ and $pd(K_4)=2$ from Theorem \ref{t2},
the result is true for $n=4$. Suppose that if $H$ is a $3$-regular graph of order $n$ $(n\geq 4)$, then $pd(G)\leq 2$.
Let $G$ be a $3$-regular graph of order $n+1$. We will show $pd(G)\leq 2$. If $G$ has no cut edge, then $pd(G)\leq 2$ from Lemma \ref{lpdnc}.
So, we consider $G$ having a cut edge, say $uv$ $(u,v\in V(G))$. We delete the cut edge $uv$, then there are two components containing $u$ and $v$,
respectively, say $G_1$, $G_2$. Since $G$ is $3$-regular, we have $|V(G_1)|\geq 5$ and $|V(G_2)|\geq 5$. Thus,
$5\leq|V(G_1)|\leq n-4$ and $5\leq|V(G_2)|\leq n-4$. Obviously, $G_1$ and $G_2$ are the graphs $H(u)$, $H(v)$, respectively.
We first show the following claims.
\begin{claim}\label{lpd2} \
$pd(H_1(u))\leq 2$.
\end{claim}
\noindent $Proof.$ Let $u_1$ and $u_2$ be two neighbors of $u$ in $H_1(u)$. Assume that the neighbors of $u_1$ and $u_2$ in $H_1(u)$ are $\{u, u_2, w_1\}$, $\{u, u_1, w_2\}$,
respectively. The edges $u_1w_1$, $u_1u_2$ and $u_2w_2$ are denoted by $e_1,e_2,e_3$. Let $A=\{u,u_1,u_2\}$ and $B=V(H_1(u))\setminus A$.
Since $|V(G_1)|\leq n-4$, we have $|V(H_1'(u))|\leq n-1$. Obviously, $H_1'(u)$ is $3$-regular. Then $pd(H_1'(u))\leq 2$ by the induction hypothesis.
Let $c'$ be a proper disconnection coloring of $H_1'(u)$ with two colors. For any two vertices $p$ and $q$ of $H_1'(u)$,
let $R_{pq}$ be a $p$-$q$ proper edge-cut of $H_1'(u)$. There are two cases to discuss.

\textbf{Case 1}. $c'(e_1)=c'(e_2)$ or $c'(e_2)=c'(e_3)$.

Without loss of generality, we assume $c'(e_1)=c'(e_2)=1$. We define an edge-coloring $c$ of $H_1(u)$ as follows.
Let $c(uu_1)=2$, $c(uu_2)=1$ and $c(e)=c'(e)$ $(e\in E(H_1(u))\setminus\{uu_1,uu_2\})$. Let $x$ and $y$ be two vertices of $H_1(u)$.
If they are both in $H_1(u)\setminus\{u\}$, then $(R_{xy}\cap E(H_1(u)))\cup\{uu_1\}$ is an $x$-$y$ proper edge-cut of $H_1(u)$.
If $x=u$ or $y=u$, then $E_u$ is an $x$-$y$ proper edge-cut of $H_1(u)$. So, $c$ is a proper disconnection coloring of $H_1(u)$.
Thus, $pd(H_1(u))\leq 2$.

\textbf{Case 2}. $c'(e_1)=c'(e_3)\neq c'(e_2)$.

Assume $c'(e_1)=c'(e_3)=1$ and $c'(e_2)=2$. Define an edge-coloring $c$ of $H_1(u)$ as follows.
Let $c(uu_1)=2$, $c(uu_2)=1$ and $c(e)=c'(e)$ $(e\in E(H_1(u))\setminus\{uu_1,uu_2\})$.
Let $x$ and $y$ be two vertices of $H_1(u)$. If $x=u$ or $y=u$, then $E_u$ is an $x$-$y$ proper edge-cut of $H_1(u)$.
If $x,y\in A\setminus\{u\}$, then $\{uu_2,e_1,e_2\}$ is an $x$-$y$ proper edge-cut of $H_1(u)$.
If $x\in A\setminus\{u\},y\in B$ or $x\in B, y\in A\setminus\{u\}$, then $\{e_1,e_3\}$ is an $x$-$y$ proper edge-cut of $H_1(u)$.
Considering $x,y\in B$, if $e_1,e_3\notin R_{xy}$, then $(R_{xy}\cap E(H_1(u)))\cup\{uu_2\}$ is an $x$-$y$ proper edge-cut of $H_1(u)$.
Otherwise, i.e., $e_1\in R_{xy}$ or $e_3\in R_{xy}$, then $R_{xy}\cap E(H_1(u))$ is an $x$-$y$ proper edge-cut of $H_1(u)$.
So, $c$ is a proper disconnection coloring of $H_1(u)$. Thus, $pd(H_1(u))\leq 2$.

\begin{claim}\label{lpd2} \
$pd(H_2(u))\leq 2$.
\end{claim}

\begin{proof}
Assume that the neighbors of $u$ in $H_2(u)$ are $u_1$ and $u_2$.
Since $|V(H_2'(u))|<|V(G_2)|$ and $H_2'(u)$ is $3$-regular,
$pd(H_2'(u))\leq 2$ by the induction hypothesis. Let $c'$ be a proper disconnection coloring of $H_2'(u)$ with two colors.
We define an edge-coloring $c$ of $H_2(u)$ as follows: $c(uu_1)=1, c(uu_2)=2$ and $c(e)=c'(e)$ $(e\in E(H_2(u))\setminus\{uu_1,uu_2\})$.
Assume $c'(u_1u_2)=c(uu_i)$ ($i=1$ or $2$).
Then for any two vertices $x$ and $y$ of $H_2(u)$, if $x=u$ or $y=u$, then $E_u$ forms an $x$-$y$ proper edge-cut.
Otherwise, assume that the $x$-$y$ proper edge-cut in $H_2'(u)$ is $R$.
If $u_1u_2\notin R$, then $R$ is an $x$-$y$ proper edge-cut.
If $u_1u_2\in R$, then $(R\cup\{uu_i\})\setminus\{u_1u_2\}$ is an $x$-$y$ proper edge-cut.
So, $c$ is a proper disconnection coloring of $H_2(u)$. Thus, $pd(H_2(u))\leq 2$.

So, from the above claims we have $pd(G_1)\leq 2$. Similarly, we have $pd(G_2)\leq 2$.
Then, there exists a proper disconnection coloring $c_0$ of $G_1\cup G_2$ with two colors.
Now we assign color $1$ to the cut edge $uv$. It is a proper disconnection coloring of $G$.
So, $pd(G)\leq 2$.
\end{proof}

A \emph{block} of a graph $G$ is a maximal connected subgraph of $G$ that has no cut vertex. It is
obvious that a block is a $K_2$ or a 2-connected subgraph with at least three vertices.
Let $\{B_1,B_2, ...,B_t\}$ be the set of blocks of $G$.

\begin{lemma}\upshape\cite{BCJ}\label{pdblock}
Let $G$ be a nontrivial connected graph. Then $\mathit{pd}(G)=\max\{ \mathit{pd}(B_i)|i=1,2,\ldots, t \}$.
\end{lemma}

\begin{theorem}\label{pd26}
If $G$ is a graph of order $n$ with maximum degree $\Delta(G)=3$, then $pd(G)\leq 2$.
Particularly, if $G$ satisfies the condition of Theorem \ref{pd1}, then $pd(G)=1$;
otherwise, $pd(G)=2$.
\end{theorem}
\begin{proof}
\begin{figure}[h]
    \centering
    \includegraphics[width=100pt]{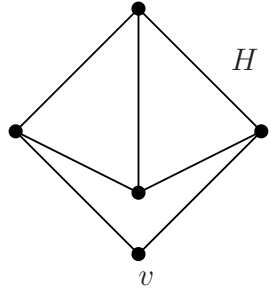}\\
    \caption{The graph H.} \label{1}
\end{figure}
If $G$ is a tree, then $pd(G)=1$ by Theorem \ref{treepd}. Suppose $G$ is not a tree. Let $H$ be a graph as shown in Figure \ref{1},
where $v$ is called the {\em key vertex} of $H$.
Suppose $G$ is a graph with maximum degree three. Let $G'$ be a graph obtained from $G$ by deleting pendent edges one by one.
Then $\Delta(G')\leq 3$ and $pd(G)=pd(G')$ by Lemma \ref{pdblock}.
Let $\{u_1,\cdots,u_t\}$ be the set of $2$-degree vertices in $G'$ and $H_1,\cdots,H_t$ be $t$ copies of $H$ such that the key vertex of
$H_i$ is $v_i$ ($i\in [t]$). We construct a new graph $G''$ obtained by connecting $v_i$ and $u_i$ for each $i\in[t]$.
Then $G''$ is a $3$-regular graph. By Lemma \ref{pd3regular}, $pd(G'')\leq2$. Since $G'$ is a subgraph of $G''$, $pd(G')\leq2$.
\end{proof}

\begin{theorem}
Let $G$ be a connected graph with maximum degree $Delta=3$ such that the set of vertices with degree $3$ in $G$ forms an independent set.
If $G$ contains a triangle or $K_{2,3}$, then $pd(G)=2$; otherwise, $pd(G)=1$.
\end{theorem}
\begin{proof}
If $G$ contains a triangle or a $K_{2,3}$, then there exist two vertices such that no matching cut separates them.
So, $pd(G)=2$ by Theorem \ref{pd26}. Now consider that $G$ is both triangle-free
and $K_{2,3}$-free. We proceed by induction on the order $n$ of $G$.
Since $\Delta(G)=3$, we have $n\geq 4$.
If $n=4$, then the graph $G$ is $K_{1,3}$ and
$pd(G)=1$ by Theorem \ref{treepd}. The result holds for $n=4$.
Assume $pd(G)=1$ for triangle-free and $K_{2,3}$-free graphs with order $n$ satisfying the condition.
Now, consider a graph $G$ with order $n+1$. Let $x$ and $y$ be two vertices of $G$.

For $d(x)=1$, the edge set $E_x$ is an $x$-$y$ matching cut.

For $d(x)=d(y)=2$, if $x$ and $y$ are adjacent, let $x_1,y_1$ be another neighbor of $x$ and $y$, respectively.
Let $G'=G-xy$. Then by the induction hypothesis, there exist an $x_1$-$y_1$ matching cut $R$ in $G'$.
Thus, $R\cup\{xy\}$ is an $x$-$y$ matching cut in $G$. If $x$ and $y$ are nonadjacent, then assume $N(x)=\{x_1,x_2\}$.
Since $G$ contains no triangles, then $x_1$ and $x_2$ are nonadjacent.
There are two cases to consider. If $d(x_1)=2$, then let $u_1$ be another neighbor of $x_1$, and then $\{xx_2,x_1u_1\}$ is an $x$-$y$ matching cut.
If $d(x_1)=d(x_2)=3$, let $N(x_1)=\{x,u_1,u_2\}$ and $N(x_2)=\{x,v_1,v_2\}$. There are two cases to consider.
If $\{u_1,u_2\}\cap \{v_1,v_2\}\neq \emptyset$, assume $u_1=v_1$.
Let $w,q$ be another neighbor of $u_2$ and $v_2$, respectively. If $y\neq v_2$, then $\{xx_1,x_2u_1,v_2q\}$ is an $x$-$y$ matching cut.
Otherwise, $\{xx_2,x_1u_1,u_2w\}$ is an $x$-$y$ matching cut. Assume $\{u_1,u_2\}\cap \{v_1,v_2\}=\emptyset$.
Let $w,q$ be another neighbor of $u_1$ and $u_2$, respectively. If $y=u_1$, then $\{xx_2,x_1u_1,u_2q\}$ is an $x$-$y$ matching cut.
Otherwise, $\{xx_2,x_1u_2,u_1w\}$ is an $x$-$y$ matching cut.

For $d(x)=3$ (or $d(y)=3$), assume $N(x)=\{x_1,x_2,x_3\}$.
Since the set of vertices with degree $3$ in $G$ forms an independent set, the neighbors of $x$ have degree at most two.
Since $G$ is $K_{2,3}$-free, there exists at least one vertex in $N(x)$ which has only one common neighbor $x$ with the others in $N(x)$.
Without loss of generality, say $x_1$.
Let $N(x_1)=\{x,s_1\},N(x_2)=\{x,s_2\}$ and $N(x_3)=\{x,s_3\}$ ($s_2=s_3$ is possible ).
If $x$ and $y$ are nonadjacent, then $\{x_1s_1,x_2s_2,xx_3\}$ is an $x$-$y$ matching cut.
If $x$ and $y$ are adjacent, there are three cases to consider.
When $y=x_2$ (or $x_3$), we have $\{x_1s_1,xy,x_3s_3\}$ (or $\{x_1s_1,xy,x_2s_2\}$) is an $x$-$y$ matching cut.
When $y=x_1$ and $s_2=s_3$, if $d(s_2)=2$, then $\{xy\}$ is an $x$-$y$ matching cut;
if $d(s_2)=3$, then assume $N(s_2)=\{x_2,x_3,p_1\}$, and then $\{xy,s_2p_1\}$ is an $x$-$y$ matching cut.
When $y=x_1$ and $s_2\neq s_3$, we have $\{xy,x_2s_2,x_3s_3\}$ is an $x$-$y$ matching cut. Thus, $pd(G)=1$ by Theorem \ref{pd1}.
\end{proof}

\begin{corollary}
Let $G$ be a connected graph with $\Delta=3$. If the set of vertices with degree $3$ in $G$ forms an independent set, then deciding whether $pd(G)=1$ is solvable in polynomial time.
\end{corollary}

Naturally, we can ask the following question.

\begin{Question}
Let $G$ be a connected graph with $\Delta=3$. Is it true that deciding whether $pd(G)=1$ is solvable in polynomial time ?
\end{Question}

\subsection{Hardness results for bipartite graphs}

\noindent In fact, by Corollary \ref{cor}, we know that it is $NP$-complete to decide whether $pd(G)=1$ for a general graph $G$.
In this subsection, we will further show that given a bipartite graph $G$, deciding whether $pd(G)=1$ is $NP$-complete.

Let $G$ be a simple connected graph. We employ the idea used in \cite{M} to construct a new graph $G^*$, which is constructed
as follows: $G^*$ is obtained from $G$ by replacing each edge by a $4$-cycle. Then $G^*$ has two types of vertices:
\emph{old} vertices, which are vertices of $G$, and \emph{new} vertices, which are not vertices of $G$.
For example, for an edge $e=uv \in E(G)$, replace it by a 4-cycle $C_{e}=uxvyu$.
Then $u,v$ are old vertices and $x,y$ are new vertices. Observe that all new vertices of $G^*$
have degree two, and each edge of $G^*$ connects an old vertex to a new vertex.
Clearly, $G^*$ is a bipartite graph with one side of the bipartition consisting only of vertices of degree $2$.

\begin{lemma}\label{bipar}
Let $G$ be a simple connected graph. Then $pd(G)=1$ if and only if $pd(G^*)=1$.
\end{lemma}

\begin{proof}
Suppose $pd(G^*)=1$. For any two vertices $x,y$ of $G$, $x,y$ are old vertices in $G^*$.
By Theorem \ref{pd1}, there exists an $x$-$y$ matching cut $F$ in $G^*$.
Then $F$ consists of pairs of matching edges in the same $4$-cycle.
Let $F'$ be the edge set obtained by replacing each pair of matching edges of $F$ in the same $4$-cycle
by the edge to which the $4$-cycle corresponds in $G$. Then $F'$ is an $x$-$y$ matching cut in $G$.

Suppose $pd(G)=1$. Then for any two vertices $u,v$ of $G$, there is a $u$-$v$ matching cut.
We denote it by $F_{uv}$. Let $F$ be an edge subset of $E(G)$. Choose two matching edges from each $4$-cycle
to which each edge of $F$ corresponds in $G^*$. Denote the edge set by $F^*$.

For any two vertices $x,y$ in $G^*$, if $x,y$ are old vertices, then $F^*_{xy}$ is an $x$-$y$ matching cut in $G^*$.
If $x$ is an old vertex and $y$ is a new vertex, there are two cases to consider.
If $x,y$ are in a same 4-cycle, assume that the 4-cycle is $xyzwx$. We know $xz\in F_{xz}$ in $G$.
Then $(F_{xz}\setminus \{xz\})^*\cup \{xy,zw\}$ is an $x$-$y$ matching cut in $G^*$.
If $x,y$ are in different 4-cycles, say $C_1=xv_1uv_2x$ and $C_2=yz_1wz_2y$. If there are no $x$-$z_2$ paths in $G-F_{xz_1}$,
then $F^*_{xz_{1}}$ is an $x$-$y$ matching cut in $G^*$.
Otherwise, we know $z_1z_2\in F_{xz_1}$.
Then $(F_{xz_1}\setminus \{z_1z_2\})^*\cup \{yz_2,wz_1\}$ is an $x$-$y$ matching cut in $G^*$.
If $x,y$ are new vertices, there are two cases to consider.
If $x,y$ are in a same 4-cycle, assume that the 4-cycle is $\{uxvyu\}$. Then $F^*_{uv}$ is an $x$-$y$ matching cut in $G^*$.
If $x,y$ are in different 4-cycles, say $C_1=u_1xu_2vu_1$ and $C_2=z_1yz_2wz_1$.
Denote the component containing $u_1$ by $C$ and the remaining part by $\bar{C}$ in $G-F_{u_1u_2}$.
If $\{z_1z_2\}\subseteq \bar{C}$, then $(F_{u_1u_2}\setminus \{u_1u_2\})^*\cup \{u_1v,u_2x\}$ is an $x$-$y$ matching cut in $G^*$.
If $z_1 \in C$ and $z_2\in \bar{C}$, then $z_1z_2\in F_{u_1u_2}$.
Then, $(F_{u_1u_2}\setminus \{u_1u_2,z_1z_2\})^*\cup \{u_1v,u_2x,z_1y,z_2w\}$ is an $x$-$y$ matching cut in $G^*$.
If $\{z_1z_2\}\subseteq C$, then $(F_{u_1u_2}\setminus \{u_1u_2\})^*\cup \{u_1x,u_2v\}$ is an $x$-$y$ matching cut in $G^*$.
\end{proof}

From the above Lemma \ref{bipar}, we can immediately get the following result.

\begin{theorem}
Given a bipartite graph $G$, deciding whether $pd(G)=1$ is $NP$-complete.
\end{theorem}

\section{Hardness results for rainbow vertex-disconnection number}

\noindent In this section, we show that it is $NP$-complete to decide
whether a given vertex-colored graph $G$ is rainbow vertex-disconnected, even though the graph $G$ has maximum degree $\Delta(G)=3$ or is bipartite.

\begin{lemma}\label{RVD}
Let $G$ be a $k$-vertex-colored graph where $k$ is a fixed positive integer. Deciding whether $G$ is rainbow vertex-disconnected under this coloring is in $P$.
\end{lemma}

\begin{proof}
Let $x$ and $y$ be any two vertices of $G$. Since $G$ is a vertex-colored graph, any rainbow vertex-cut $S$ have no more than $k$ vertices.
There are at most $n-2 \choose k$ choices for $S$, which is a polynomial of $n$ for a fixed $k$. For any two nonadjacent (or adjacent) vertices $x$, $y$ of $G$, it is polynomial time to
check whether $x$ and $y$ are in different components of $G-S$ (or $(G-xy)-S$). There are at most $n \choose 2$ pairs of vertices in $G$. Thus,
it is polynomial time to deciding whether $G$ is rainbow vertex-disconnected.

\end{proof}

\begin{theorem}\label{hard}
Let $G$ be a vertex-colored graph and $s$ and $t$ be two vertices of $G$. Deciding whether there is a rainbow vertex-cut between $s$ and $t$ is NP-complete.
\end{theorem}

\begin{proof}
This problem is NP from Lemma \ref{RVD}. We now show that the problem is NP-complete by giving a polynomial reduction from the $3$-SAT problem to this problem.
Given a 3CNF formula $\phi=\wedge_{i=1}^{m} c_{i}$ over $n$ variables $x_{1},x_{2},\cdots,x_{n}$, we construct a graph $G_{\phi}$ with two special vertices $s,t$
and a vertex-coloring $f$ such that there is a rainbow vertex-cut between $s,t$ in $G_{\phi}$ if and only if $\phi$ is satisfied. Let $\theta_{c_{i}}(x_j)$ denote
the location of literal $x_j$ in clause $c_i$ for $i\in[m]$ and $j\in[n]$.

We define $G_{\phi}$ as follows:
\begin{align*}
V(G_{\phi})& =\{c_i,u_{i,k},v_{i,k},w_{i,k}: i\in[m],k\in[3] \}\cup \{x_j,\bar{x}_j: j\in[n] \} \cup \{s,t\}. \\
E(G_{\phi})& =\{x_j u_{i,k},\bar{x}_j w_{i,k}:\mbox{If $x_{j}\in c_i$ and $\theta_{c_{i}}(x_j)=k$}, i\in [m],j\in[n],k\in\{1,2,3\}\} \\
& \cup \{x_j w_{i,k},\bar{x}_j u_{i,k}:\mbox{If $\bar{x}_{j}\in c_i$ and $\theta_{c_{i}}(\bar{x}_j)=k$},i\in [m],j\in[n],k\in\{1,2,3\}\} \\
& \cup \{u_{i,k} v_{i,k}: i\in [n], k\in \{1,2,3\}\}   \cup
\{sx_j,s\bar{x}_j:j\in[m]\}   \\
&  \cup \{c_iv_{i,k},c_iw_{i,k}: i\in [n], k\in \{1,2,3\}\}   \cup \{tc_i: i\in [n]\}  \\
&  \cup \{st\}.
\end{align*}

Now we define a vertex-coloring $f$ of $G_\phi$ as follows. For $i\in[m],j\in[n]$ and $k\in [3]$, let $f(x_j)=f(\bar{x}_j)=r_{j}$,
$f(w_{i,k})=r_{i,k}$, $f(u_{i,k})=r_{i,4}$, $f(v_{i,k})=r_{i,5}$, $f(s)=f(t)=f(c_i)=r$. All those colors are distinct.

\begin{figure}[h]
    \centering
    \includegraphics[width=265pt]{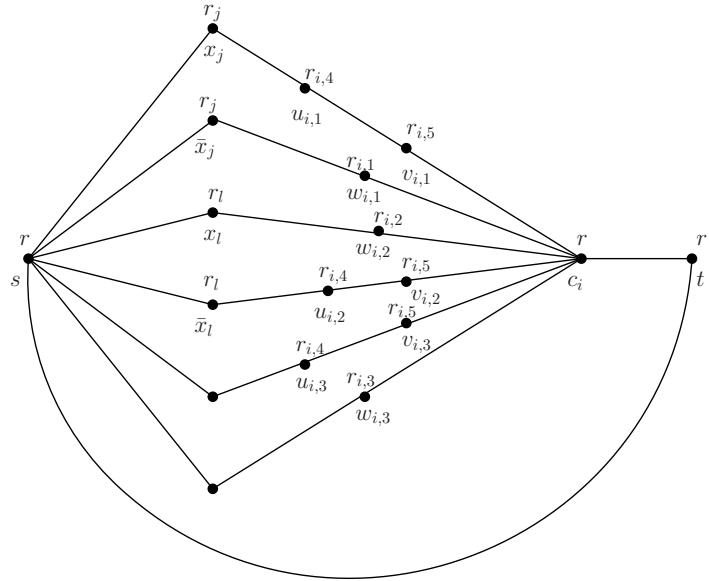}\\
    \caption{The variables $x_j,\bar{x_l}\in c_i$ and $x_j,\bar{x_l}$ are the first and second literature respectively.} \label{1}
\end{figure}

We claim that there is a rainbow vertex-cut between $s$ and $t$ in $G_{\phi}$ if and only if $\phi$ is satisfied.

Suppose that there is a rainbow $s$-$t$ vertex-cut $S$ in $G_\phi$. Since $s$ and $t$ are adjacent in $G_\phi$,
$S+s$ or $S+t$ is rainbow and so $c_i \notin S$ for $i\in [n]$.
Thus $S$ also separates $s$ and $c_i$. Note that there are three $s$-$c_{i}$ paths of length $4$. Since $f(u_{i,k})=r_{i,4}$ and $f(v_{i,k})=r_{i,5}$ for $k\in [3]$,
there exists at least one $j\ (j\in [n])$ such that $x_{j} \in S$ or $\bar{x}_{j} \in S$. Since $f(x_j)=f(\bar{x}_j)=r_{j}$, $x_j$ and $\bar{x}_{j}$ can not belong to $S$ simultaneously.
If $x_{j} \in S$, set $x_j=1$. If $\bar{x}_{j} \in S$, set $x_j=0$. Then the literature associated with $x_j$ in clause $c_{i}$ is satisfied and $c_i$ is true. Since $S$ is a rainbow $s$-$t$ vertex-cut, there are no conflicts on the truth assignments of the variables. Therefore, $\phi$ is satisfied.

Suppose that $\phi$ is satisfied. We now try to find an $s$-$t$ rainbow vertex-cut $S$ in $G_{\phi}$ under the coloring $f$. Since $f(s)=f(t)=f(c_i)=r$ and  $s$, $t$ are adjacent, then $c_i \notin S$.  For any variable $x_j(j\in [n])$, if $x_j=0$, let the vertex $\bar{x}_j \in S$. In this case, if $x_j\in c_i$, then $x_j$ is adjacent to $u_{i,k}$ in $G_\phi$ and let one vertex of $\{u_{i,k},v_{i,k}\}$ belong to $S$ for $i\in [m],j\in[n],k\in\{1,2,3\}$. If $\bar{x}_j\in c_i$, then $x_j$ is adjacent to $w_{i,k}$ in $G_\phi$ and let vertex $\{w_{i,k}\}\in S$ for $i\in [m],j\in[n],k\in\{1,2,3\}$. For any variable $x_j\ (j\in [n])$, if $x_j=1$, let the vertex $x_j \in S$. In this case, if $x_j\in c_i$, then let vertex $\{w_{i,k}\}\in S$ for $i\in [m],j\in[n],k\in\{1,2,3\}$. If $\bar{x}_j\in c_i$, then let one vertex of $\{u_{i,k},v_{i,k}\}$ belong to $S$ for $i\in [m],j\in[n],k\in\{1,2,3\}$. By the choice of $S$, we know that if a literal of $c_i$ is false, then a vertex-colored with $r_{i,4}$ or $r_{i,5}$ is in $S$. So if two literals of some clause $c_i$ are false, we put two vertices colored with $r_{i,4}$ and $r_{i,5}$ respectively to $S$. Since each clause $c_i$ is satisfied, the vertex set $S$ is rainbow. Thus $S$ is an $s$-$t$ rainbow vertex-cut.

\end{proof}

\begin{theorem}
Let $G$ be a vertex-colored graph with maximum degree $\Delta=3$ and $s$ and $t$ be two vertices of $G$. Then deciding whether there is a rainbow vertex-cut between $s$ and $t$ is $NP$-complete.
\end{theorem}

\begin{proof}
Based on the vertex-colored graph $G_\phi$ in Theorem \ref{hard}, we can obtain a new graph $G_\phi^{*}$ by doing the following operation on $G_\phi$. We change each vertex $v$ with degree more than $3$ to a cycle $C_v$ with $d(v)$ new vertices. The new vertices in the cycle will connect the neighbors of $v$, respectively. We color all the new vertices of $C_v$ using the same color with $v$. If a new vertex $v_1$ from $C_s$ connects the new vertex $v_2$ from $C_t$, we regard $v_1$ as a new vertex $s$ and $v_2$ as a new vertex $t$. Similarly to the proof of Theorem \ref{hard}, we can prove that deciding whether there is a rainbow vertex-cut between $s$ and $t$ in graph $G_\phi^{*}$ is $NP$-complete.
\end{proof}

\begin{theorem}
Let $G$ be a vertex-colored bipartite graph and $s$ and $t$ be two vertices of $G$. Deciding whether there is a rainbow vertex-cut between $s$ and $t$ is NP-complete.
\end{theorem}

\begin{proof}
By Theorem \ref{hard}, we know that there is a rainbow vertex-cut between $s$ and $t$ in $G_{\phi}$ if and only if $\phi$ is satisfied. Construct a graph $G'_{\phi}$ by subdividing all edges of $G_{\phi}$. Then assign the new vertices with color $r$ and the other vertices with the same color as in $G_{\phi}$. It is easy to show that there is a rainbow vertex-cut between $s$ and $t$ in $G'_{\phi}$ if and only if $\phi$ is satisfied. The proof is thus
complete.
\end{proof}

\noindent {\bf Remark:} This paper is an extended version of \cite{CLLW}, which was published in the proceedings of FAW 2020, Lecture Notes in Computer Science, a LNCS number has not been given, yet.


\begin{thebibliography}{8}

\bibitem{ALLP} Andrews, E., Laforge, E., Lumduanhom, C., Zhang, P.: On proper-path colorings in graphs. J. Combin. Math. Combin. Comput. \textbf{97}, 189--207 (2016)

\bibitem{BCL} Bai, X., Chang, R., Huang, Z., Li, X.: More on rainbow disconnection in graphs. Discuss. Math. Graph Theory. doi:10.7151/dmgt.2333, in press

\bibitem{BCJ} Bai, X., Chen, Y., Ji, M., Li, X., Weng, Y., Wu, W.: Proper disconnection in graphs. arXiv:1906.01832 [math.CO]

\bibitem{BCLLW} Bai, X., Chen, Y., Li, X., Li, P., Weng, Y.: The rainbow vertex-disconnection in graphs. accepted for publication in Acta Math. Sin. (Engl. Ser.); arXiv:1812.10034 [math.CO].

\bibitem{BHL} Bai, X., Huang, Z., Li, X.: Bounds for the rainbow disconnection number of graphs. arXiv:2003.13237 [math.CO]

\bibitem{BM} Bondy, J. A., Murty, U. S. R.: Graph Theory. Springer $2008$

\bibitem{BFGMMMT} Borozan, V., Fujita, S., Gerek, A., Magnant, C.,  Manoussakis, Y., Montero, L., Tuza, Zs.: Proper connection of graphs. Discrete Math., \textbf{312}, 2550--2560 (2012)

\bibitem{CDHHZ} Chartrand, G., Devereaux, S., Haynes, T. W., Hedetniemi, S. T., Zhang, P.: Rainbow disconnection in graphs. Discuss. Math. Graph Theory, \textbf{38}, 1007--1021 (2018)

\bibitem{CJMZ} Chartrand, G., Johns, G. L., McKeon, K. A., Zhang, P.: Rainbow connection in graphs. Math. Bohem., \textbf{133}, 85--98 (2008)

\bibitem{CLLW} Chen, Y., Li, P., Li, X., Weng, Y.: Complexity results for the proper disconnection of graphs, Proceedings of 14th International Frontiers of Algorithmics Workshop
(FAW 2020), LNCS No. Haikou, Hainan, 29¨C31 May, 2020

\bibitem{ADJD} Darmann, A., D\"ocker, J.: On simplified $NP$-complete variants of Not-All-Equal 3-SAT and 3-SAT. arXiv:1908.04198 [cs.CC]

\bibitem{LW} Li, X., Weng, Y.: Further results on the rainbow vertex-disconnection of graphs. arXiv:2004.06285 [math.CO]

\bibitem{M} Moshi, A. M.: Matching cutsets in graphs. J. Graph Theory, \textbf{13}, 527--536 (1989)

\bibitem{PP} Patrignani, M., Pizzonia, M.: The complexity of the matching-cut problem. In: Brandst\"adt A., Le V.B. (eds) Graph-Theoretic Concepts in Computer Science. WG 2001. LNCS, vol. 2204, pp. 284--295. Springer, Heidelberg (2001). \doi{10.1007/3-540-45477-2\_26}

\bibitem{sch78} Schaefer, T. J.: The complexity of satisfiability problems. In: Procceedings of the 10th annual ACM sympodium on Theory of Computing, pp. 216--226. ACM, New York (1978). \doi{10.1145/800133.804350}

\end{thebibliography}
\end{document}